\documentclass{amsart}
\usepackage[all]{xy}
\usepackage{color}
\usepackage{stmaryrd}
\usepackage{amsthm}
\usepackage{amssymb}
\usepackage[colorlinks=true]{hyperref}%to generate  P D F  with links using latex

\numberwithin{equation}{subsection}

\newtheorem{theorem}{Theorem}[section]
\newtheorem*{theorem*}{Theorem}
\newtheorem{lemma}[theorem]{Lemma}
\newtheorem{proposition}[theorem]{Proposition}
\newtheorem{corollary}[theorem]{Corollary}
\newtheorem*{corollary*}{Corollary}

\theoremstyle{remark}
\newtheorem{definition}[theorem]{Definition}
\theoremstyle{remark}
\newtheorem{example}[theorem]{Example}

\theoremstyle{remark}
\newtheorem{remark}[theorem]{Remark}
\theoremstyle{remark}
\newtheorem{notation}[theorem]{Notation}
\newcommand{\too}{\longrightarrow}% long rightarrow
\newcommand{\cC}{{\mathcal C}}

\newcommand{\cT}{{\mathcal T}}

\newcommand{\bbA}{\mathbb{A}}

\newcommand{\bbQ}{\mathbb{Q}}
\newcommand{\bbP}{\mathbb{P}}

\newcommand{\bbZ}{\mathbb{Z}}

\newcommand{\ie}{\textsl{i.e.}\ }
\newcommand{\eg}{\textsl{e.g.}}

\newcommand{\internalcomment}[1]{}

\title[Exponentiation of motivic measures]{Exponentiation of motivic measures}
\author{Niranjan Ramachandran and Gon{\c c}alo~Tabuada}

\address{Niranjan Ramachandran, Department of Mathematics, University of Maryland, College Park, MD 20742 USA.}
\email{atma@math.umd.edu}
\urladdr{http://www2.math.umd.edu/~atma/}

\address{Gon{\c c}alo Tabuada, Department of Mathematics, MIT, Cambridge, MA 02139, USA}
\email{tabuada@math.mit.edu}
\urladdr{http://math.mit.edu/~tabuada}

\date{\today}

\begin{document}
\begin{abstract}
In this short note we establish some properties of all those motivic measures which can be exponentiated. As a first  application, we show that the rationality of Kapranov's zeta function is stable under products. As a second application,  we give an elementary proof of a result of Totaro. \end{abstract}
\maketitle

\vskip-\baselineskip
\vskip-\baselineskip
\vskip-\baselineskip
%\vskip-\baselineskip
%\tableofcontents
%----------------------------------------------------
\section{Motivic measures}
%----------------------------------------------------
Let $k$ be an arbitrary base field and $\mathrm{Var}(k)$ the category of {\em varieties}, \ie reduced separated $k$-schemes of finite type. The {\em Grothendieck ring of varieties $K_0\mathrm{Var}(k)$} is defined as the quotient of the free abelian group on the set of isomorphism classes of varieties $[X]$ by the relations $[X]=[Y]+[X\backslash Y]$, where $Y$ is a closed subvariety of $X$. The multiplication is induced by the product over $\mathrm{Spec}(k)$. When $k$ is of positive characteristic, one needs also to impose the relation $[X]=[Y]$ for every surjective radicial morphism $X \to Y$; see Musta\c{t}\v{a} \cite[Page~78]{mustata}. Let ${\bf L}:=[\bbA^1]$.

The structure of the Grothendieck ring of varieties is quite mysterious; see Poonen \cite{Poonen} for instance. In order to capture some of its flavor several {\em motivic measures}, \ie ring homomorphisms $\mu:K_0\mathrm{Var}(k) \to R$, have been built. Examples include the counting measure $\mu_\#$ (see \cite[Ex.~7.7]{mustata}); the Euler characteristic measure $\chi_c$ (see \cite[Ex.~7.8]{mustata}); the Hodge characteristic measure $\mu_{\mathrm{H}}$ (see \cite[\S4.1]{LS}); the Poincar{\'e} characteristic measure $\mu_{\mathrm{P}}$ with values in $\bbZ[u]$ (see \cite[\S4.1]{LS}); the Larsen-Lunts ``exotic'' measure $\mu_{\mathrm{LL}}$ (see \cite{Larsen}); the Albanese measure $\mu_{\mathrm{Alb}}$ with values in the semigroup ring of isogeny classes of abelian varieties (see \cite[Thm.~7.21]{mustata}); the Gillet-Soul{\'e} measure $\mu_{\mathrm{GS}}$ with values in the Grothendieck ring $K_0(\mathrm{Chow}(k)_\bbQ)$ of Chow motives (see \cite{GS}); and the measure $\mu_{\mathrm{NC}}$ with values in the Grothendieck ring of noncommutative Chow motives (see \cite{CvsNC}). There exist several relations between the  above motivic measures. For example, $\chi_c, \mu_{\mathrm{H}}, \mu_{\mathrm{P}}, \mu_{\mathrm{NC}}$, factor through $\mu_{\mathrm{GS}}$.
%----------------------------------------------------
\section{Kapranov's zeta function}
%----------------------------------------------------
As explained in \cite[Prop.~7.27]{mustata}, in the construction of the Grothendieck ring of varieties we can restrict ourselves to quasi-projective varieties. Given a motivic measure $\mu$, Kapranov introduced in \cite{Kapranov} the associated zeta function 
\begin{equation}\label{eq:zeta}
\zeta_\mu(X;t) := \sum_{n=0}^\infty \mu([S^n(X)])t^n \in (1+ R\llbracket t \rrbracket)^\times\,,
\end{equation}
where $S^n(X)$ stands for the $n^{th}$ symmetric product of the quasi-projective variety $X$. In the particular case of the counting measure, \eqref{eq:zeta} agrees with the classical Weil zeta function. Here are some other computations (with $X$ smooth projective)
\begin{eqnarray*}
\zeta_{\chi_c}(X;t)=(1-t)^{-\chi_c(X)} &\!\! \zeta_{\mathrm{P}}(X;t)=\prod_{r\geq 0} (\frac{1}{1- u^r t})^{(-1)^{b_r}} &\!\! \zeta_{\mathrm{Alb}}(X;t) = \frac{[\mathrm{Alb}(X)]t}{1-t}\,,
\end{eqnarray*}
where $b_r:= \mathrm{dim}_{\mathbb C} H^r_{dR}(X)$ and $\mathrm{Alb}(X)$ is the Albanese variety of $X$; see \cite[\S3]{Ramachandran}.
%----------------------------------------------------
\section{Big Witt ring}
%----------------------------------------------------
Given a commutative ring $R$, recall from Bloch \cite[Page~192]{Bloch} the construction of the big Witt ring $W(R)$. As an additive group, $W(R)$ is equal to $((1+R\llbracket t \rrbracket)^\times, \times)$. Let us write $+_W$ for the addition in $W(R)$ and $1=1+0t+\cdots$ for the zero element. The multiplication $\ast$ in $W(R)$ is uniquely determined by the following~requirements:
\begin{itemize}
\item[(i)] The equality $(1-at)^{-1} \ast (1-bt)^{-1} = (1-abt)^{-1}$ holds for every $a, b \in R$;
\item[(ii)] The assignment $R \mapsto W(R)$ is an endofunctor of commutative rings.
\end{itemize} 
The unit element is $(1-t)^{-1}$. We have also a (multiplicative) Teichm\"uller map
\begin{eqnarray*}
R \too W(R) && a \mapsto [a]:=(1-at)^{-1}
\end{eqnarray*}
such that $g(t)\ast [a]=g(at)$ for every $a \in R$ and $g(t) \in W(R)$; see \cite[Page~193]{Bloch}. \begin{definition}
Elements of the form $p(t)-_Wq(t) \in W(R)$, with $p(t),q(t) \in R[t]$ and $p(0)=q(0)=1 \in R$, are called {\em rational functions}. 
\end{definition}
Let $W_{\mathrm{rat}}(R)$ be the subset of rational elements. As proved by Naumann in \cite[Prop.~6]{Naumann}, $W_{\mathrm{rat}}(R)$ is a subring of $W(R)$. Moreover, $R\mapsto W_{\mathrm{rat}}(R)$ is an endofunctor of commutative rings. Recall also the construction of the commutative ring $\Lambda(R)$. As an additive group, $\Lambda(R)$ is equal to $W(R)$. The multiplication is uniquely determined by the requirement that the involution group isomorphism $\iota: \Lambda(R) \to W(R), g(t) \mapsto g(-t)^{-1}$, is a ring isomorphism. The unit element is~$1+t$.
%We recall the ring $\Lambda(R)$, a variant of $W(R)$. The bijective map $\iota: f(t) \mapsto f(-t)^{-1}$ is an involution of the group $((1+R\llbracket t \rrbracket)^\times, \times)$.  The ring $\Lambda(R)$, with underlying group $((1+R\llbracket t \rrbracket)^\times, \times)$, is uniquely characterized by the requirement that $\iota:W(R) \to \Lambda(R)$ is a ring isomorphism.  
%----------------------------------------------------
\section{Exponentiation}
%---------------------------------------------------- 
Let $\mu$ be a motivic measure. As explained by  Musta\c{t}\v{a} in \cite[Prop. 7.28]{mustata}, the assignment $X \mapsto \zeta_\mu(X;t)$ gives rise to a group homomorphism
\begin{equation}\label{zeta-Kapranov}
\zeta_\mu(-;t): K_0\mathrm{Var}(k) \too W(R)\,.
\end{equation}
%It is therefore natural to introduce the following notion:
\begin{definition}{(\cite[\S3]{Ramachandran})}\label{def:expo} A motivic measure $\mu$ {\em can be exponentiated}\footnote{Note that Kapranov's zeta function is similar to the exponential function $e^x = \sum_{n=0}^{\infty} \frac{x^n}{n!}$. The product $X^n$ corresponds to $x^n$ and the symmetric product $S^n(X)$ corresponds to $\frac{x^n}{n!}$ since $n!$ is the size of the symmetric group on $n$ letters.} if the above group homomorphism \eqref{zeta-Kapranov} is a ring homomorphism. 
\end{definition}
\begin{corollary}\label{cor:main}
Given a motivic measure $\mu$ as in Definition \ref{def:expo}, the following~holds:
\begin{itemize}
\item[(i)] The ring homomorphism \eqref{zeta-Kapranov} is a new motivic measure;
\item[(ii)] Any motivic measure which factors through $\mu$ can also be exponentiated.
\end{itemize}
\end{corollary}
This class of motivic measures is well-behaved with respect with rationality:
\begin{proposition}\label{prop:product}
Let $\mu$ be a motivic measure as in Definition \ref{def:expo}. If $\zeta_\mu(X;t)$ and $\zeta_\mu(Y;t)$ are rational functions, then $\zeta_\mu(X\times Y;t)$ is also a rational function.
\end{proposition}
\begin{proof}
It follows automatically from the fact that $W_{\mathrm{rat}}(R)$ is a subring of $W(R)$.
\end{proof}
As proved by Naumann in \cite[Prop.~8]{Naumann} (see also \cite[Thm.~2.1]{Ramachandran}), the counting measure $\mu_\#$ can be exponentiated. On the other hand, Larsen-Lunts ``exotic'' measure $\mu_{\mathrm{LL}}$ {\em cannot} be exponentiated! This would imply, in particular, that
\begin{equation}\label{eq:product-curves}
\zeta_{\mu_{\mathrm{LL}}}(C_1\times C_2;t) = \zeta_{\mu_{\mathrm{LL}}}(C_1;t) \ast \zeta_{\mu_{\mathrm{LL}}}(C_2;t)
\end{equation} 
for any two smooth projective curves $C_1$ and $C_2$. As proved by Kapranov in \cite{Kapranov} (see also \cite[Thm.~7.33]{mustata}), $\zeta_\mu(C;t)$ is a rational function for every smooth projective curve $C$ and motivic measure $\mu$. Using Proposition \ref{prop:product}, this hence implies that the right-hand side of \eqref{eq:product-curves} is also a rational function. On the other hand, as proved by Larsen-Lunts in \cite[Thm.~7.6]{Larsen}, the left-hand side of \eqref{eq:product-curves} is not a rational function whenever $C_1$ and $C_2$ have positive genus. We hence obtain a contradiction. 

At this point, it is natural to ask which motivic measures can be exponentiated? We now provide a general answer to this question using the notion of $\lambda$-ring. Recall that a {\em $\lambda$-ring $R$} consists of a commutative ring equipped with a sequence of maps $\lambda^n:A \to A, n \geq 0$, such that $\lambda^0(a)=1$, $\lambda^1(a)=a$, and $\lambda^n(a+b)=\sum_{i +j =n} \lambda^i(a) \lambda^j(b)$ for every $a, b \in R$. In other words, the map
\begin{eqnarray*}
\lambda_t: R \too \Lambda(R) && a \mapsto \lambda_t(a):= \sum_n \lambda^n(a) t^n
\end{eqnarray*}
is a group homomorphism. Equivalently, the composed map 
\begin{eqnarray}\label{eq:lambda-1}
\sigma_t: R \xrightarrow{\lambda_t} \Lambda(R) \xrightarrow{\iota} W(R) && a \mapsto \sigma_t(a):=\lambda_{-t}(a)^{-1}
\end{eqnarray}
is a group homomorphism. This homomorphism is called the {\em opposite} $\lambda$-structure.
\begin{proposition}\label{prop:expo}
Let $\mu$ be a motivic measure and $R$ a $\lambda$-ring such that:
\begin{itemize}
\item[(i)] The above group homomorphism \eqref{eq:lambda-1} is a ring homomorphism;
\item[(ii)] We have $\mu([S^n(X)])=\sigma^n(\mu([X]))$ for every quasi-projective variety $X$.
\end{itemize}
Under these conditions, the motivic measure $\mu$ can be exponentiated.
\end{proposition}
\begin{proof}
Consider the following composed ring homomorphism
\begin{equation}\label{eq:comp}
K_0\mathrm{Var}(k)\stackrel{\mu}{\too} R \stackrel{\sigma_t}{\too} W(R)\,.
\end{equation}
The equalities $\mu([S^n(X)])=\sigma^n(\mu([X]))$ allow us to conclude that \eqref{eq:comp} agrees with the group homomorphism $\zeta_\mu(-;t)$. This achieves the proof.
\end{proof}
%Let us now make some remarks on conditions (i)-(ii) of Proposition \ref{prop:expo}.
\begin{remark}\label{rk:1}
Let $\cC$ be a $\bbQ$-linear additive idempotent complete symmetric monoidal category. As proved by Heinloth in \cite[Lem.~4.1]{heinloth}, the exterior powers give rise to a special $\lambda$-structure on the Grothendieck ring $K_0(\cC)$, with opposite $\lambda$-structure given by the symmetric powers $\mathrm{Sym}^n$. In this case, \eqref{eq:lambda-1} is a ring homomorphism.
\end{remark}
\begin{remark}\label{rk:2}
Let $\cT'$ be a $\bbQ$-linear thick triangulated monoidal subcategory of compact objects in the homotopy category $\cT=\mathrm{Ho}(\cC)$ of a simplicial symmetric monoidal model category $\cC$. As proved by Guletskii in \cite[Thm.~1]{Guletskii}, the exterior powers give rise to a special $\lambda$-structure on $K_0(\cT')$, with opposite $\lambda$-structure given by the symmetric powers $\mathrm{Sym}^n$. In the case, \eqref{eq:lambda-1} is a ring homomorphism.
\end{remark}

%\begin{example}
%For any ring $R$, the ring $\Lambda(R)$ is a $\lambda$-ring. The maps $\lambda^n:\Lambda(R) \to \Lambda(R)$ are characterized by  \cite[4.1]{Larsen} $$\lambda^n \prod_{i=1}^{N} (1+a_jt) = \underset{S\subset \{1,\cdots, N\},,  \#S =n }{\prod} (1 + t\prod_{j\in S} a_j). $$
% \Goncalo{[Explain here the $\lambda$-structure on $W(R)$]}
%\end{example}

%A $\lambda$-structure is called {\em special} if \eqref{eq:lambda} is a ring homomorphism which preserves the $\lambda$-structures of $R$ and $\Lambda(R)$. A {\em (special) $\lambda$-ring $R$} consists of a commutative ring $R$ equipped with a (special)~$\lambda$-structure.  For any special $\lambda$-ring $R$, both the maps $\lambda_t: R \to \Lambda(R)$ and $\sigma_{t}: R \to W(R)$ are  ring homomorphisms; we call these special $\lambda$-structures which are opposite to each other.
\begin{remark}\label{rk:3}
Assume that $k$ is of characteristic zero. Thanks to Heinloth's presentation of the Grothendieck group of varieties (see \cite[Thm.~3.1]{Bittner}), it suffices to verify the equality $\mu([S^n(X)])=\sigma^n(\mu([X]))$ for every smooth projective variety $X$.
\end{remark}
As an application of the above Proposition \ref{prop:expo}, we obtain the following result:
\begin{proposition}\label{prop:GS}
The Gillet-Soul{\'e} motivic measure $\mu_{\mathrm{GS}}$ can be exponentiated.
\end{proposition}
\begin{proof}
Recall from \cite{GS} that $\mu_{\mathrm{GS}}$ is induced by the symmetric monoidal functor 
\begin{equation}\label{eq:functor-h}
\mathfrak{h}: \mathrm{SmProj}(k) \too \mathrm{Chow}(k)_\bbQ
\end{equation}
from the category of smooth projective varieties to the category of Chow motives. Since the latter category is $\bbQ$-linear, additive, idempotent complete, and symmetric monoidal, Remark \ref{rk:1} implies that the Grothendieck ring $K_0(\mathrm{Chow}(k)_\bbQ)$ satisfies condition (i) of Proposition \ref{prop:expo}. As proved by del Ba\~no-Aznar in \cite[Cor.~2.4]{Del}, we have $\mathfrak{h}(S^n(X))\simeq \mathrm{Sym}^n \mathfrak{h}(X)$ for every smooth projective variety $X$. Using Remark \ref{rk:3}, this hence implies that condition (ii) of Proposition \ref{prop:expo} is also satisfied.
\end{proof}
\begin{remark}
Thanks to Corollary \ref{cor:main}(ii), all the motivic measures which factor through $\mu_{\mathrm{GS}}$ (\eg\ $\chi_c, \mu_{\mathrm{H}}, \mu_{\mathrm{P}},\mu_{\mathrm{NC}}$) can also be exponentiated.
\end{remark}
%----------------------------------------------------
\section{Application I: rationality of zeta functions}
%----------------------------------------------------
By combining Propositions \ref{prop:product} and \ref{prop:GS}, we obtain the following result:
\begin{corollary}\label{cor:GS}
Let $X, Y$ be two varieties. If $\zeta_{\mu_{\mathrm{GS}}}(X;t)$ and $\zeta_{\mu_{\mathrm{GS}}}(Y;t)$ are rational functions, then $\zeta_{\mu_{\mathrm{GS}}}(X\times Y;t)$ is also a rational function.
\end{corollary}
\begin{remark}
Corollary \ref{cor:GS} was independently obtained by Heinloth \cite[Prop.~6.1]{heinloth} in the particular case of smooth projective varieties and under the extra assumption that $\zeta_{\mu_{\mathrm{GS}}}(X;t)$ and $\zeta_{\mu_{\mathrm{GS}}}(Y;t)$ satisfy a certain functional equation.
\end{remark}
\begin{example}\label{ex:1}
Let $X, Y$ be smooth projective varieties (\eg\ abelian varieties) for which $\mathfrak{h}(X),\mathfrak{h}(Y)$ are Kimura-finite; see \cite[\S3]{Kimura}. Consider the ring homomorphism
\begin{equation}\label{eq:ring-1}
\sigma_t: K_0(\mathrm{Chow}(k)_\bbQ)\too W(K_0(\mathrm{Chow}(k)_\bbQ))\,.
\end{equation}
As proved by Andr{\'e} in \cite[Prop.~4.6]{Andre}, $\sigma_t([\mathfrak{h}(X)])$ and $\sigma_t([\mathfrak{h}(Y)])$ are rational functions. Since $\zeta_{\mu_{\mathrm{GS}}}(-;t)$ agrees with the composition of $\mu_{\mathrm{GS}}$ with $\eqref{eq:ring-1}$, these latter functions are equal to $\zeta_{\mu_{\mathrm{GS}}}(X;t)$ and $\zeta_{\mu_{\mathrm{GS}}}(Y;t)$, respectively. Using Corollary \ref{cor:GS}, we hence conclude that $\zeta_{\mu_{\mathrm{GS}}}(X\times Y;t)$ is also a rational function. %Moreover, this latter rational function can be ``explicitly'' described as $\zeta_{\mu_{\mathrm{GS}}}(X;t)\ast \zeta_{\mu_{\mathrm{GS}}}(Y;t)$.
\end{example}
Recall from Voevodsky \cite[\S2.2]{Voevodsky} the construction of the functor 
\begin{equation}\label{eq:functor-mix}
M^c: \mathrm{Var}(k)^p \too \mathrm{DM}_{\mathrm{gm}}(k)_\bbQ
\end{equation}
from the category of varieties and proper morphisms to the triangulated category of geometric motives. As proved in \cite[Prop.~4.1.7]{Voevodsky}, the functor \eqref{eq:functor-mix} is symmetric monoidal. Moreover, given a variety $X$ and a closed subvariety $Y \subset X$, we have
$$ M^c(Y) \too M^c(X) \too M^c(X\backslash Y) \too M^c(Y)[1]\,;$$
see \cite[Prop.~4.1.5]{Voevodsky}. Consequently, we obtain the following motivic measure 
\begin{eqnarray}\label{eq:measure-mix}
K_0\mathrm{Var}(k) \too K_0(\mathrm{DM}_{\mathrm{gm}}(k)_\bbQ) && [X] \mapsto [M^c(X)]\,.
\end{eqnarray}
\begin{proposition}\label{prop:extension}
The above motivic measure \eqref{eq:measure-mix} agrees with $\mu_{\mathrm{GS}}$.
\end{proposition}
\begin{proof}
As proved by Voevodsky in \cite[Prop.~2.1.4]{Voevodsky}, there exists a $\bbQ$-linear additive fully-faithful symmetric monoidal functor
\begin{equation}\label{eq:functor}
\mathrm{Chow}(k)_\bbQ \too \mathrm{DM}_{\mathrm{gm}}(k)_\bbQ
\end{equation}
such that $\eqref{eq:functor}\circ \mathfrak{h}(X) \simeq M^c(X)$ for every smooth projective variety. Thanks to the work of Bondarko \cite[Cor.~6.4.3 and Rk.~6.4.4]{Bondarko-1}, the above functor \eqref{eq:functor} induces a ring isomorphism $K_0(\mathrm{Chow}(k)_\bbQ)\simeq K_0(\mathrm{DM}_{\mathrm{gm}}(k)_\bbQ)$. Therefore, the proof follows from Heinloth's presentation of the Grothendieck ring of varieties in terms of smooth projective varieties; see \cite[Thm.~3.1]{Bittner}.
\end{proof}
Thanks to Proposition \ref{prop:extension}, Example \ref{ex:1} admits the following generalization:
\begin{example}\label{ex:2}
Let $X,Y$ be varieties for which $M^c(X),M^c(Y)$ are Kimura-finite. Similarly to Example \ref{ex:1}, $\zeta_{\mu_{\mathrm{GS}}}(X\times Y;t)$ is then a rational function.
\end{example}
In the above Examples \ref{ex:1} and \ref{ex:2}, the rationality of $\zeta_{\mu_{\mathrm{GS}}}(X\times Y;t)$ can alternatively be deduced from the stability of Kimura-finiteness under tensor products; see \cite[\S5]{Kimura}. Thanks to the work of O'Sullivan-Mazza \cite[\S5.1]{Mazza} and Guletskii \cite{Guletskii}, the above Corollary \ref{cor:GS} can also be applied to non Kimura-finite situations.
\begin{proposition}\label{prop:others}
Let $X_0$ be a connected smooth projective surface over an algebraically closed field $k_0$ such that $q=0$ and $p_g >0$, $k:=k_0(X_0)$ the function field of $X_0$, $x_0$ a $k_0$-point of $X_0$, $z$ the zero-cycle which is the pull-back of the cycle $\Delta(X_0)-(x_0\times X)$ along $X_0 \times k \to X_0 \times X_0$, $Z$ the support of $z$, and finally $U$ the complement of $Z$ in $X=X_0 \times k$. Under these notations, the following holds:
\begin{itemize}
\item[(i)] The geometric motive $M^c(U)$ is not Kimura-finite;
\item[(ii)] Kapranov's zeta function $\zeta_{\mu_{\mathrm{GS}}}(U;t)$ is rational.
\end{itemize}
\end{proposition}
\begin{proof}
As proved by O'Sullivan-Mazza in \cite[Thm.~5.18]{Mazza}, $M(U)$ is not Kimura-finite. Since the surface $U$ is smooth, we have $M^c(U)\simeq M(U)^\ast(2)[4]$; see \cite[Thm.~4.3.7]{Voevodsky}. Using the fact that $-(2)[4]$ is an auto-equivalence of the category $\mathrm{DM}_{\mathrm{gm}}(k)_\bbQ$ and that $M(U)^\ast$ is Kimura-finite if and only if $M(U)$ is Kimura-finite (see Deligne \cite[Prop.~1.18]{Deligne}), we conclude that $M^c(U)$ also is not~Kimura-finite. 

We now prove item (ii). As proved by Guletskii in \cite[\S3]{Guletskii}, the category $\mathrm{DM}_{\mathrm{gm}}(k)_\bbQ$ satisfies the conditions of Remark \ref{rk:2}. Consequently, we have a ring homomorphism
\begin{equation}\label{eq:Lambda-1}
\sigma_t: K_0(\mathrm{DM}_{\mathrm{gm}}(k)_\bbQ) \too W(K_0(\mathrm{DM}_{\mathrm{gm}}(k)_\bbQ))\,.
\end{equation}
As explained by Guletskii in \cite[Ex. 5]{Guletskii}, $\sigma_t([M(U)])$ is a rational function. Thanks to Lemma \ref{lem:aux-11} below, we hence conclude that $\sigma_t([M^c(U)])$ is also a rational function. The proof follows now from the fact that $\zeta_{\mu_{\mathrm{GS}}}(-;t)$ agrees with the composition of the ring homomorphisms \eqref{eq:measure-mix} and \eqref{eq:Lambda-1}.
\end{proof}
\begin{lemma}\label{lem:aux-11}
Given a smooth variety $X$ of dimension $d$, we have the equality
$$\sigma_t([M^c(X)])= \sigma_{\mu_{\mathrm{GS}}({\bf L})^dt}([M(X)])\,.$$
\end{lemma}
\begin{proof}
The proof is given by the following identifications
\begin{eqnarray}
\sigma_t([M^c(X)]) & = & \sigma_t([M(X)^\ast(d)[2d]]) \label{eq:star0} \\
& = & \sigma_t([M(X)^\ast]\mu_{\mathrm{GS}}({\bf L}^d)) \nonumber \\
& = & \sigma_t([M(X)^\ast])\ast \zeta_{\mu_{\mathrm{GS}}}({\bf L}^d;t) \nonumber \\
& = & \sigma_t([M(X)])\ast \zeta_{\mu_{\mathrm{GS}}}({\bf L}^d;t) \label{eq:star1} \\
& = & \sigma_{\mu_{\mathrm{GS}}({\bf L})^dt}([M(X)])\,, \label{eq:star2}
\end{eqnarray}
where \eqref{eq:star0} follows from \cite[Thm.~4.3.7]{Voevodsky}, \eqref{eq:star1} from \cite[Lem.~1.18]{Deligne}, and \eqref{eq:star2} from Remark \ref{rk:last} below with $\mu:=\mu_{\mathrm{GS}}$ and $g(t):=\sigma_t([M(X)])$.  %proof of Proposition \ref{prop:Totaro} below with $\zeta_{\mu}(X;t)$ replaced by $\sigma_t([M(U)])$.
\end{proof}
\begin{example}\label{ex:3}
Let $U_1,U_2$ be two surfaces as in Proposition \ref{prop:others}. Thanks to the above Corollary \ref{cor:GS}, we hence conclude that $\zeta_{\mu_{\mathrm{GS}}}(U_1\times U_2;t)$ is a rational function.
\end{example}
\begin{remark}
Thanks to Corollary \ref{cor:main}(ii), the above Examples \ref{ex:1}, \ref{ex:2}, and \ref{ex:3}, hold {\em mutatis mutandis} for any motivic measure which factors through~$\mu_{\mathrm{GS}}$. %(\eg\ $\chi_c, \mu_{\mathrm{H}}, \mu_{\mathrm{P}},\mu_{\mathrm{NC}}$). 
\end{remark}
%----------------------------------------------------
\section{Application II: Totaro's result}
%----------------------------------------------------
The following result plays a central role in the study of the zeta~functions.
\begin{proposition}[Totaro]\label{prop:Totaro}
The equality $\zeta_\mu(X\times \bbA^n;t) = \zeta_\mu(X; \mu({\bf L})^nt)$ holds for every variety $X$ and motivic measure $\mu$.
\end{proposition}
Its proof (see \cite[Lem.~4.4]{Goet}\cite[Prop.~7.32]{mustata}) is non-trivial and based on a stratification of the symmetric products of $X \times \bbA^n$. In all the cases where the motivic measure $\mu$ can be exponentiated, this result admits the following elementary proof:
\begin{proof}
Since $[X\times \bbA^n]=[X][\bbA^n]$ in the Grothendieck ring of varieties and the motivic measure $\mu$ can be exponentiated, the proof is given by the identifications
\begin{eqnarray}
\zeta_\mu(X\times \bbA^n;t) & = & \zeta_\mu(X;t) \ast \zeta_\mu({\bf L}^n;t) \nonumber\\
& = &  \zeta_\mu(X;t) \ast \zeta_\mu({\bf L};t)^{\ast n} \nonumber \\
 & = & \zeta_\mu(X;t) \ast (1 + \mu({\bf L})t + \mu({\bf L})t^2+ \cdots )^{\ast n} \label{eq:equality-1}\\
 & = & \zeta_\mu(X;t) \ast ((1-\mu({\bf L})t)^{-1})^{\ast n}\nonumber\\
 & = & \zeta_\mu(X;t) \ast [\mu({\bf L})]^{\ast n} \nonumber\\
& = & \zeta_\mu(X;t) \ast [\mu({\bf L})^n] \nonumber\\
& = & \zeta_\mu(X;\mu({\bf L})^nt)\,,\nonumber
\end{eqnarray}
where \eqref{eq:equality-1} follows from \cite[Ex.~7.23]{mustata} and $[\mu({\bf L})]$ stands for the image of $\mu({\bf L}) \in R$ under the multiplicative Teichm\"uller map $R \to W(R)$.
\end{proof}
\begin{remark}\label{rk:last}
The above proof shows more generally that $g(t) \ast \zeta_\mu({\bf L}^n;t)=g(\mu({\bf L})^nt)$ for every $g(t) \in W(R)$ and motivic measure $\mu$ which can be exponentiated.
\end{remark}
\begin{remark}{(Fiber bundles)}
Given a fiber bundle $E\to X$ of rank $n$, we have $[E]=[X][\bbA^n]$ in the Grothendieck ring of varieties; see \cite[Prop.~7.4]{mustata}. Therefore, the above proof, with $X$ replaced by $E$, shows  that $\zeta_\mu(E;t) = \zeta_\mu(X;\mu({\bf L})^nt)$.
%for every motivic measure $\mu$ which can be exponentiated. 
\end{remark}
\begin{remark}{($\bbP^n$-bundles)}
Given a $\bbP^n$-bundle $E \to X$, we have $[E]=[X][\bbP^n]$ in the Grothendieck ring of varieties; see \cite[Ex.~7.5]{mustata}. Therefore, by combining the equality $[\bbP^n]=1+ {\bf L} + \cdots + {\bf L}^n$ with the above proof, we conclude that 
$$\zeta_\mu(E;t) = \zeta_\mu(X;t) +_W \zeta_\mu(X;\mu({\bf L})t) +_W \cdots +_W \zeta_\mu(X;\mu({\bf L})^nt)\,.$$
%for every motivic measure $\mu$ which can be exponentiated.  
\end{remark}
%----------------------------------------------------
\section{$G$-varieties}
%----------------------------------------------------
Let $G$ be a finite group and $\mathrm{Var}^G(k)$ the category of {\em $G$-varieties}, \ie varieties $X$ equipped with a $G$-action $\lambda: G \times X \to X$ such that every orbit is contained in an affine open set. The {\em Grothendieck ring of $G$-varieties} $K_0\mathrm{Var}^G(k)$ is defined as the quotient of the free abelian group on the set of isomorphism classes of $G$-varieties $[X,\lambda]$ by the relations $[X,\lambda]=[Y,\tau]+[X\backslash Y,\lambda]$, where $(Y,\tau)$ is a closed $G$-invariant subvariety of $(X,\lambda)$. The multiplication is induced by the product of varieties. A motivic measure is a ring homomorphism $\mu^G: K_0\mathrm{Var}^G(k) \to R$. As mentioned in \cite[\S5]{Looijenga}, the above measures $\chi_c, \mu_{\mathrm{H}}, \mu_{\mathrm{P}}$ admit $G$-extensions $\chi_c^G, \mu_{\mathrm{H}}^G, \mu_{\mathrm{P}}^G$.
\begin{notation}
Let $\mathrm{Chow}^G(k)_\bbQ$ be the category of functors from the group $G$ (considered as a category with a single object) to the category $\mathrm{Chow}(k)_\bbQ$.
\end{notation}
Note that $\mathrm{Chow}^G(k)_\bbQ$ is still a $\bbQ$-linear additive idempotent complete symmetric monoidal category and that \eqref{eq:functor-h} extends to a symmetric monoidal functor
\begin{equation}\label{eq:functor-G}
\mathfrak{h}^G:\mathrm{SmProj}^G(k) \too \mathrm{Chow}^G(k)_\bbQ\,.
\end{equation}
Note also that the $n^{\mathrm{th}}$ symmetric product of a $G$-variety is still a $G$-variety. Therefore, the notion of exponentiation makes sense in this generality. Gillet-Soul{\'e}'s motivic measure $\mu_{\mathrm{GS}}$ admits the following $G$-extension:
\begin{proposition}\label{prop:new}
The above functor \eqref{eq:functor-G} gives rise to a motivic measure
$$ \mu_{\mathrm{GS}}^G: K_0\mathrm{Var}^G(k) \too K_0(\mathrm{Chow}^G(k)_\bbQ)$$
which can be exponentiated.
\end{proposition}
\begin{proof}
Given a smooth projective variety $X$ and a closed subvariety $Y$, let us denote by $\mathrm{Bl}_Y(X)$ the blow-up of $X$ along $Y$ and by $E$ the associated exceptional divisor. As proved by Manin in \cite[\S9]{Manin}, we have a natural isomorphism $\mathfrak{h}(\mathrm{Bl}_Y(X))\oplus \mathfrak{h}(Y) \simeq \mathfrak{h}(X)\oplus \mathfrak{h}(E)$ in $\mathrm{Chow}(k)_\bbQ$. Since this isomorphism is natural, it also holds in $\mathrm{Chow}^G(k)_\bbQ$ when $X$ is replaced by a smooth projective $G$-variety $(X,\lambda)$ and $Y$ by a closed $G$-invariant subvariety $(Y,\tau)$. Therefore, thanks to Heinloth's presentation of the Grothendieck ring of $G$-varieties in terms of smooth projective $G$-varieties (see \cite[Lem.~7.1]{Bittner}), the assignment $X \mapsto \mathfrak{h}^G(X)$ gives rise to a (unique) motivic measure $\mu_{\mathrm{GS}}^G$. The proof of Proposition \ref{prop:GS}, with \eqref{eq:functor-h} replaced by \eqref{eq:functor-G}, shows that this motivic measure $\mu^G_{\mathrm{GS}}$ can be exponentiated. 
\end{proof}
\begin{remark}
Thanks to Corollary \ref{cor:main}(ii), all the motivic measures which factor through $\mu^G_{\mathrm{GS}}$ (\eg\ $\chi_c^G, \mu_{\mathrm{H}}^G, \mu^G_{\mathrm{P}}$) can also be exponentiated.
\end{remark}
Proposition \ref{prop:product} admits the following $G$-extension:
\begin{proposition}\label{prop:GS1}
Let $\mu^G$ be a motivic measure which can be exponentiated and $(X,\lambda), (Y,\tau)$ two $G$-varieties. If $\zeta_{\mu^G}((X,\lambda);t)$ and $\zeta_{\mu^G}((Y,\tau);t)$ are rational functions, then $\zeta_{\mu^G}((X\times Y, \lambda \times \tau);t)$ is also a rational function.
\end{proposition}
\begin{example}
Assume that the group $G$ (of order $r$) is abelian and that the base field $k$ is algebraically closed of characteristic zero or of positive characteristic $p$ with $p \nmid r$. Under these assumptions, Mazur proved in \cite[Thm.~1.1]{Mazur} that $\zeta_{\mu^G}((C,\lambda);t)$ is a rational function for every smooth projective $G$-curve $(C,\lambda)$ and motivic measure $\mu^G$. Thanks to Proposition \ref{prop:GS1}, we hence conclude that $\zeta_{\mu^G}((C_1\times C_2,\lambda_1 \times \lambda_2);t)$ is still a rational function for every motivic measure $\mu^G$ which can be exponentiated and for any two smooth projective $G$-curves $(C_1,\lambda_1)$ and $(C_2,\lambda_2)$.
\end{example}
Finally, Totaro's result admits the following $G$-extension:
\begin{proposition}
Let $\mu^G$ be a motivic measure which can be exponentiated and $(X.\lambda), (\bbA^n,\tau)$ two $G$-varieties. When $G$ (of order $r$) is abelian and $k$ is algebraically closed, Kapranov's zeta function $\zeta_{\mu^G}((X\times \bbA^n,\lambda\times \tau);t)$ agrees with 
$$ \zeta_{\mu^G}\left((X,\lambda); \mu^G(S^r(\bbA^n,\tau))t\right)+_W \zeta_{\mu^G}((X,\lambda);t)\ast \left(\sum_{l=0}^{r-1}\prod_{i=1}^n \mu^G([\bbA^1,\tau_i]\cdots[\bbA^1,\tau_i^l])t^l\right)\,,$$
where $[\bbA^n,\tau]=[\bbA^1,\tau_1]\cdots[\bbA^1,\tau_n]$.
\end{proposition}
\begin{proof}
Since $[X\times \bbA^n,\lambda\times \tau]=[X,\lambda][\bbA^n,\tau]$ in the Grothendieck ring of $G$-varieties and the motivic measure $\mu^G$ can be exponentiated, we have the equality
\begin{equation*}\label{eq:equality-last-1}
\zeta_{\mu^G}((X\times \bbA^n,\lambda\times \tau);t) = \zeta_{\mu^G}((X,\lambda); t) \ast \zeta_{\mu^G}((\bbA^n,\tau); t)\,.
\end{equation*}
Moreover, as explained in \cite[Page~1338]{Mazur}, we have the following computation
\begin{equation*}\label{eq:equality-last-2}
\zeta_{\mu^G}((\bbA^n,\tau); t) = \frac{1}{1-\mu^G(S^r(\bbA^n,\tau))t}\left(\sum_{l=0}^{r-1}\prod_{i=1}^n \mu^G([\bbA^1,\tau_i]\cdots[\bbA^1,\tau_i^l])t^l\right)\,.
\end{equation*}
Therefore, since $(1-\mu^G(S^r(\bbA^n,\tau))t)^{-1}$ is the Teichm\"uller class $[\mu^G(S^r(\bbA^n,\tau))]$, the proof follows from the combination of the above equalities.
\end{proof}

\end{document}